\documentclass{amsart}

\usepackage{amsmath, amssymb, amsthm, amsfonts}

\input xy
\xyoption{all}

\usepackage{hyperref}

\swapnumbers
\numberwithin{equation}{section}

\theoremstyle{plain}
\newtheorem{theorem}[subsection]{Theorem}
\newtheorem{lemma}[subsection]{Lemma}
\newtheorem{prop}[subsection]{Proposition}

\theoremstyle{definition}
\newtheorem{defn}[subsection]{Definition}

\newtheorem{remark}[subsection]{Remark}
\newtheorem{exam}[subsection]{Example}
\newtheorem{ques}[subsection]{Question}

\def\AA{\mathbb{A}}

\def\CC{\mathbb{C}}

\def\FF{\mathbb{F}}
\def\GG{\mathbb{G}}

\def\QQ{\mathbb{Q}}
\def\RR{\mathbb{R}}

\def\ZZ{\mathbb{Z}}

\newcommand\cC{\mathcal{C}}
\newcommand\cD{\mathcal{D}}

\newcommand\cF{\mathcal{F}}

\newcommand\cH{\mathcal{H}}

\newcommand\cL{\mathcal{L}}

\def\bD{\mathbf{D}}

\def\bO{\mathbf{O}}

\def\bT{\mathbf{T}}

\newcommand\fra{\mathfrak{a}}

\newcommand\fo{\mathfrak{o}}

\newcommand\tilW{\widetilde{W}}

\newcommand{\can}{\textup{can}}

\newcommand\id{\textup{id}}

\newcommand\pt{\textup{pt}}

\newcommand\res{\textup{res}}

\newcommand\rs{\textup{rs}}

\newcommand\Span{\textup{Span}}

\newcommand\Aut{\textup{Aut}}
\newcommand\Hom{\textup{Hom}}

\newcommand\SL{\textup{SL}}

\newcommand{\Gm}{\GG_m}

\newcommand{\Ad}{\textup{Ad}}

\newcommand\xch{\mathbb{X}^*}
\newcommand\xcoch{\mathbb{X}_*}

\newcommand{\incl}{\hookrightarrow}
\newcommand{\isom}{\stackrel{\sim}{\to}}

\newcommand{\Qlbar}{\overline{\QQ}_\ell}

\renewcommand{\j}[1]{\langle{#1}\rangle}
\newcommand{\wt}[1]{\widetilde{#1}}
\newcommand{\wh}[1]{\widehat{#1}}
\newcommand\quash[1]{}
\newcommand\mat[4]{\left(\begin{array}{cc} #1 & #2 \\ #3 & #4 \end{array}\right)}  
\newcommand\un{\underline}

\newcommand{\ov}{\overline}
\newcommand{\bs}{\backslash}

\newcommand\xr{\xrightarrow}

\newcommand\ot{\otimes}

\newcommand{\cohog}[2]{\textup{H}^{#1}({#2})}     
\newcommand{\cohoc}[2]{\textup{H}_{c}^{#1}({#2})}     

\renewcommand\a\alpha
\renewcommand\b\beta
\newcommand\g\gamma
\renewcommand\d\delta
\newcommand\D\Delta
\newcommand{\e}{\epsilon}
\renewcommand{\th}{\theta}

\newcommand{\ph}{\varphi}
\newcommand{\s}{\sigma}
\renewcommand{\t}{\tau}

\newcommand{\y}{\eta}
\newcommand{\z}{\zeta}
\newcommand{\ep}{\epsilon}

\renewcommand{\l}{\lambda}
\renewcommand{\L}{\Lambda}
\newcommand{\om}{\omega}
\newcommand{\Om}{\Omega}

\newcommand\nb{\nabla}

\renewcommand{\c}{\circ}
\newcommand{\dw}{\dot{w}}

\title{Higher signs for Coxeter groups}

\dedicatory{}
\author{Zhiwei Yun}
\thanks{Partially supported by the Packard Foundation.}
\address{Department of Mathematics, Massachusetts Institute of Technology, 77 Massachusetts Ave, Cambridge, MA 02139}
\email{zyun@mit.edu}
\date{}
\subjclass[2010]{20F55, 14M15}
\keywords{}

\begin{document}

\begin{abstract}
We define and study cocycles on a Coxeter group in each degree generalizing the sign function. When the Coxeter group is a Weyl group, we explain how the degree three cocycle arises naturally from geometry representation theory.
\end{abstract}

\maketitle

\tableofcontents

\section{Introduction}
Throughout these notes, we fix a Coxeter group $(W,S)$ with length function $|\cdot|:W\to\ZZ_{\ge0}$.  The group homomorphism $\e:W\to \{\pm1\}$ given by $w\mapsto (-1)^{|w|}$ is called the {\em sign} for $W$. The sign of $W$ can be viewed as a $1$-cocycle of $W$ valued in the group $\{\pm1\}\cong\FF_{2}$. In these notes, for every positive integer $n$,  we define and study a canonical $n$-cocycle of $W$  valued in $\ZZ$ or $\FF_{2}$ depending on the parity of $n$.

The main result is the following.
\begin{theorem}\label{th:sign} Let $(W,S)$ be a Coxeter group. For any even (resp. odd) integer $n\ge1$, there is a unique $n$-cocycle $\e^{W}_{n}\in Z^{n}(W,\ZZ)$ (resp. $\e^{W}_{n}\in Z^{n}(W,\FF_{2})$)  satisfying two conditions
\begin{enumerate}
\item If $n\ge2$ and $(x_{1},x_{2},\cdots, x_{n})\in W^{n}$ satisfy $|x_{i}x_{i+1}|=|x_{i}|+|x_{i+1}|$ for some $1\le i\le n-1$, then $\e^{W}_{n}(x_{1},\cdots, x_{n})=0$. 
\item For any $s\in S$, $\e^{W}_{n}(s,s,\cdots, s)=1$.
\end{enumerate}
\end{theorem}

We will construct a cocycle universal among those satisfying the first condition above (called {\em collapsing})  (see Theorem \ref{th:univ}).   We will give formulas for computing  these cocycles in \S\ref{s:signs}. 

We will study $\e^{W}_{3}$ in more details and explain how $\e^{W}_{3}$ shows up naturally in geometric representation theory (see \S\ref{s:G}). It is natural to ask:
\begin{ques}
Do the higher cocycles $\e^{W}_{n}$ ($n\ge4$) show up in  geometric representation theory (e.g., Hecke categories)?
\end{ques}

\subsection{Notation}\label{ss:not} For a set $X$, $\ZZ[X]$ denotes the free abelian group with basis $\{[x]|x\in X\}$.

For an integer $n$, we denote
\begin{equation*}
\e(n)=\begin{cases} +, & \mbox{ if $n$ is even}; \\ -, & \mbox{ if $n$ is odd}.\end{cases}
\end{equation*}

When $W$ is a Weyl group, we adopt the following notation. Let $G$ be semisimple group over an algebraically closed field $k$ with a Borel subgroup $B$ and maximal  torus $T\subset B$ such that the corresponding Weyl group is $W$ with simple reflections $S$ given by $B$. 

\section{Universal collapsing cocycle}

\begin{defn} Let $A$ be a $\ZZ[W]$-module and $n\ge2$ be an integer. A cocycle $\z\in Z^{n}(W,A)$ is called {\em collapsing} if, for any $x_{1},\cdots ,x_{n}\in W$, whenever $|x_{i}x_{i+1}|=|x_{i}|+|x_{i+1}|$ for some $1\le i\le n-1$, we have $\z(x_{1},\cdots,x_{n})=0$.
\end{defn}

\begin{exam}[Tits section] Suppose $(W,S)$ is a Weyl group (see \S\ref{ss:not} for notation). Tits defines a section to $\pi: N_{G}(T)\to W$ as follows. For any $s\in S$, let $\ph_{s}:  \SL_{2}\to G$ be the unique homomorphism that restricts to $\a^{\vee}:\Gm\to T$ on the diagonal torus $\Gm$ of $\SL_{2}$. Let $\t(s)=\ph_{s}(\mat{0}{1}{-1}{0})$. For $w\in W$ with reduced expression $w=s_{1}\cdots s_{m}$, Tits proved that $\t(w):=\t(s_{1})\cdots\t(s_{m})$ is independent of the reduced expression. The map $\t: w\mapsto \t(w)$ defines a section to $\pi$. The map $\D: W^{2}\ni (x,y)\mapsto \t(x)\t(y)\t(xy)^{-1}\in T[2]$ defines a collapsing $2$-cocycle on $W$ valued in the two-torsion points $T[2]$ of $T$.
\end{exam}

The goal of this section is to describe the universal object among pairs $(A,\z)$ where $A$ is  a $\ZZ[W]$-module and $\z$ is a collapsing $n$-cocycle of $W$ valued in $A$.

\subsection{Walls, chambers, etc.}
We recall the geometric realization of a Coxeter group $(W,S)$ following \cite[Ch. V, \S4]{B}. 

Let $E=\RR^{S}$ equipped with basis $\{e_{s}\}_{s\in S}$ and let $E^{*}$ be its dual. There is a faithful representation of $W$ on $E^{*}$ such that $s\in S$ acts by a reflection. Let $H_{s}\subset E^{*}$ be the hyperplane fixed by $s$ (equivalently $H_{s}$ is the kernel of $e_{s}$). A {\em wall} is a hyperplane in $E^{*}$ of the form $w\cdot H_{s}$ for some  $w\in W, s\in S$. Let $\cH$ be the set of walls; it is in bijection with reflections in $W$ (the $W$-conjugates of $S$). 

%

For $H\in \cH$,  $E^{*}-H$ has two connected components called {\em reflecting half-spaces}. Let $\cD$ be the set of reflecting half-spaces. Let $C_{0}\subset E^{*}$ be open cone defined by $C=\{x\in E^{*}|\j{x,e_{s}}>0, \forall s\in S\}$.  The $W$-translates of $C_{0}$ are called {\em chambers}. Let $\cC$ be the set of chambers. 

There is an involution $\s: \cD\to \cD$ sending $D$ to  $E^{*}-\ov{D}$.  For each $H\in \cH$, let $\{D_{H}^{+},D^{-}_{H}\}$ be the two connected components of $E^{*}-H$, where $D^{+}_{H}$ contains $C_{0}$. For $H=H_{s}$, we denote $D^{\pm}_{H_{s}}$ by $D_{s}^{\pm}$.

\subsection{} For each $D\in \cD$ and $C\in \cC$, let
\begin{equation*}
v_{D}(C)=\begin{cases} 1, & \textup{ if }C\subset D; \\ 0, & \textup{ if }C\subset \s(D).\end{cases}
\end{equation*}

The free abelian group $\ZZ[\cD]$ carries an action of the involution $\s$.  Let $\ZZ[\cD]^{+}$ (resp. $\ZZ[\cD]^{-}$) be the subgroup of $\ZZ[\cD]$ consisting of those elements such that $\s(a)=a$ (resp. $\s(a)=-a$). Then $\ZZ[\cD]^{\pm}$ is the free abelian group with basis $[D^{+}_{H}]\pm[D^{-}_{H}]$ for $H\in \cH$.  Both $\ZZ[\cD]^{\pm}$ are $\ZZ[W]$-modules.

%

\begin{defn} Let $n\in\ZZ_{\ge1}$. Let $Z^{W}_{n}$ be the function
\begin{eqnarray}
\notag Z^{W}_{n}: W^{n}&\to& \ZZ[\cD]\\
\label{Zn} \un x=(x_{1},\cdots ,x_{n})&\mapsto & \sum_{D\in \cD}\left(\prod_{i=0}^{n-1}(v_{D}(C_{i})-v_{D}(C_{i+1}))\right)[D].
\end{eqnarray}
Here  $C_{i}=x_{1}\cdots x_{i}(C)$ for $1\le i\le n$.
\end{defn}
From the definition, we see that $[D]$ appears in $Z^{W}_{n}(\un x)$ if and only if  the sequence of chambers $(C_{0},\cdots, C_{n})$ is {\em $H$-alternating}  in the sense that  $C_{i}$ and $C_{i+1}$ lie on different sides of $H$, for all $0\le i\le n-1$.  The following alternative formula for $Z_{n}^{W}$ is immediate from the definition.

\begin{lemma} For $(x_{1},\cdots ,x_{n})\in W^{n}$, let $C_{i}=x_{1}\cdots x_{i}C_{0}$. Then 
\begin{equation}\label{Zn'}
Z_{n}^{W}(x_{1},\cdots ,x_{n})=\sum_{\mbox{$H\in\cH$; $(C_{0},\cdots, C_{n})$ is $H$-alternating}}(-1)^{[n/2]}[D^{+}_{H}]+(-1)^{[(n+1)/2]}[D^{-}_{H}].
\end{equation}
\end{lemma}

\begin{prop}\label{p:closed} The function $Z^{W}_{n}$ is an $n$-cocycle valued in $\ZZ[\cD]^{\e(n)}$, i.e., $Z^{W}_{n}\in Z^{n}(W, \ZZ[\cD]^{\e(n)})$.
\end{prop}
\begin{proof} By \eqref{Zn'}, $Z_{n}(x_{1},\cdots, x_{n})$ is a linear combination of $[D_{H}^{+}]+[D_{H}^{-}]$ if $n$ is even (so that $[n/2]=[(n+1)/2]$) and otherwise a linear combination of $[D_{H}^{+}]-[D_{H}^{-}]$.  This shows that $Z^{W}_{n}$ takes values in $\ZZ[\cD]^{\e(n)}$.

Now we show $Z^{W}_{n}$ is a cocycle. Write $Z^{W}_{n}$ simply as $Z_{n}$. Let $x_{1},\cdots, x_{n}, x_{n+1}\in W$. Then
\begin{equation*}
\d Z_{n}(x_{1},\cdots, x_{n+1})=x_{1}Z_{n}(x_{2},\cdots,x_{n+1})+\sum_{j=1}^{n}(-1)^{j}Z_{n}(x_{1},\cdots ,x_{j}x_{j+1},\cdots ,x_{n+1})+(-1)^{n+1}Z_{n}(x_{1},\cdots ,x_{n}).
\end{equation*}
Fix $D\in \cD$, let $h_{j}=v_{D}(C_{j})$. For $1\le j\le n$, the  coefficient of $[D]$ in $Z_{n}(x_{1},\cdots ,x_{j}x_{j+1},\cdots ,x_{n+1})$ is $(h_{0}-h_{1})\cdots (h_{j-1}-h_{j+1})\cdots (h_{n}-h_{n+1})$. The coefficient of $[D]$ in the last term $Z_{n}(x_{1},\cdots ,x_{n})$ is by definition $(h_{0}-h_{1})\cdots (h_{n-1}-h_{n})$. The coefficient of $[D]$ in $x_{1}Z_{n}(x_{2},\cdots,x_{n+1})$ is the same as the coefficient of $[x^{-1}_{1}D]=:[D']$ in $Z_{n}(x_{2},\cdots,x_{n+1})$, which is the product of $v_{D'}(C_{0})-v_{D'}(x_{2}C_{0}), v_{D'}(x_{2}C_{0})-v_{D'}(x_{2}x_{3}C_{0}),\cdots, v_{D'}(x_{2}\cdots x_{n}C_{0})-v_{D'}(x_{2}\cdots x_{n+1}C_{0})$. Since $v_{D'}(x_{2}\cdots x_{i}C_{0})=v_{D}(x_{1}\cdots x_{i}C_{0})=v_{D}(C_{i})=h_{i}$, we see that the coefficient of $[D]$ in $x_{1}Z_{n}(x_{2},\cdots,x_{n+1})$ is $(h_{1}-h_{2})\cdots (h_{n}-h_{n+1})$.  Therefore, the coefficient of $[D]$ in $\d Z_{n}(x_{1},\cdots, x_{n+1})$ is
\begin{equation}\label{hh}
(h_{1}-h_{2})\cdots (h_{n}-h_{n+1})+\sum_{j=1}^{n}(-1)^{j}(h_{0}-h_{1})\cdots (h_{j-1}-h_{j+1})\cdots (h_{n}-h_{n+1}) +(-1)^{n+1} (h_{0}-h_{1})\cdots (h_{n-1}-h_{n}).
\end{equation}
Write $d_{j}=h_{j}-h_{j+1}$, then  $h_{j-1}-h_{j+1}=d_{j-1}+d_{j}$. Expand the RHS into the $d_{j}$'s, we get each monomial $d_{0}d_{1}\cdots \wh{d}_{j}\cdots d_{n}$ appearing twice with opposite signs, for all  $0\le j\le n$. Therefore \eqref{hh} is zero, hence $\d Z_{n}=0$. 
\end{proof}


\begin{lemma}\label{l:two prop}
The cocycle $Z^{W}_{n}$ satisfies the following two properties:
\begin{enumerate}
\item If $n\ge2$ then $Z^{W}_{n}$ is collapsing.
\item For any $s\in S$, $Z^{W}_{n}(s,s,\cdots, s)=(-1)^{[n/2]}[D^{+}_{s}]+(-1)^{[(n+1)/2]}[D^{-}_{s}]$.
\end{enumerate}
\end{lemma}
\begin{proof}
(1) If $|x_{i}x_{i+1}|=|x_{i}|+|x_{i+1}|$, then for any $H\in \cH$, the three chambers $C_{i-1}=x_{1}\cdots x_{i-1}C_{0}$, $C_{i}$ and $C_{i+1}$ either lie on the same side of $H$, or $H$ separates $C_{i-1}$ with $C_{i}$ and $C_{i+1}$, or $H$ separates $C_{i-1}, C_{i}$ with $C_{i+1}$. In any case $H$ is not $(C_{i-1},C_{i},C_{i+1})$-alternating. Therefore $Z^{W}_{n}(x_{1},\cdots ,x_{n})=0$ in this case. 

(2)For $s\in S$, $H_{s}$ is the only wall separating $C_{0}$ and $sC_{0}$, so only $[D^{\pm}_{s}]$ appear in $Z^{W}_{n}(s,s,\cdots, s)$. The coefficients follow from \eqref{Zn'}.
\end{proof}

The main result of this section is that $Z_{n}^{W}$ is the universal collapsing $n$-cocycle on $W$.

\begin{theorem}\label{th:univ} Let $n\ge2$ be an integer. Let $A$ be a $\ZZ[W]$-module and $\z\in Z^{n}(W,A)$ be a collapsing cocycle. Then there exists a unique $W$-equivariant homomorphism $a:\ZZ[\cD]^{\e(n)}\to A$ such that $\z=a(Z_{n}^{W})$. 
\end{theorem}

To prove the theorem, we need two lemmas.
\begin{lemma}\label{l:as}
Suppose $n\ge2$ and $\z\in Z^{n}(W,A)$ is collapsing.  For $s\in S$ let $a_{s}:=\z(s,s,\cdots, s)\in A$. Let $s,s'\in S$ and let $w\in W$ be such that $wsw^{-1}=s'$, then
\begin{equation*}
wa_{s}=\begin{cases}a_{s'} & \mbox{ if $wD_{s}^{+}=D_{s'}^{+}$;}\\ (-1)^{n}a_{s'} & \mbox{ if $wD_{s}^{+}=D_{s'}^{-}$.}\end{cases}
\end{equation*}
\end{lemma}
\begin{proof} We have $ws=s'w$. Consider the case $wD^{+}_{s}=D_{s'}^{+}$, then $wC_{0}$ and $C_{0}$ lie on the same side of $H_{s'}$, which implies $|ws|=|s'w|>|w|$.

The cocycle condition $(\d \z)(w,s,\cdots ,s)=0$ together with the collapsing of $\z$ (applied to the first two elements $(w,s)$) imply that in the expansion of $(\d \z)(w,s,\cdots ,s)$, the only nonzero terms are $w\z(s,\cdots ,s)=wa_{s}$ and $-\z(ws,s,\cdots, s)$ (here we use $n\ge2$). Therefore $wa_{s}=\z(ws,s,\cdots, s)=\z(s'w,s,\cdots, s)$. Then we apply the cocycle condition $(\d\z)(s',w,s,\cdots ,s)=0$ to get $\z(s'w,s,\cdots, s)=\z(s',ws,s,\cdots, s)=\z(s',s'w,s,\cdots,s)$. Then we apply $(\d\z)(s',s',w,s,\cdots ,s)=0$, etc. In this way we get
\begin{equation*}
wa_{s}=\z(s'w,s,\cdots,s)=\z(s',s'w,s,\cdots ,s)=\cdots=\z(s',\cdots ,s'w).
\end{equation*}
Finally we apply the cocycle condition $(\d\z)(s',\cdots ,s',w)=0$ to conclude $\z(s',\cdots ,s'w)=\z(s',\cdots, s')=a_{s'}$. Combined with the above equality we get $wa_{s}=a_{s'}$.

When $wD^{+}_{s}=D_{s'}^{-}$, we write $w=w's$, then $w'sw'^{-1}=s'$ and $w'D^{+}_{s}=wD_{s}^{-}=D_{s'}^{+}$. From the case already proven, we get $w'a_{s}=a_{s'}$. It remains to show that $sa_{s}=(-1)^{n}a_{s}$ for then $wa_{s}=w'sa_{s}=(-)^{n}w'a_{s}=(-1)^{n}a_{s'}$. For this we apply $(\d\z)(s,s,\cdots,s)=0$ and using collapsing we see $(\d\z)(s,s,\cdots,s)=s\z(s,\cdots,s)+(-1)^{n+1}\z(s,\cdots ,s)=0$ hence $sa_{s}=(-1)^{n}a_{s}$ holds.  
\end{proof}

\begin{lemma}\label{l:xi van}
Let $\xi\in Z^{n}(W,A)$ be a collapsing $n$-cocycle if $n\ge2$ or satisfying $\xi(1)=0$ if $n=1$. If $\xi(s,\cdots ,s)=0$ for all $s\in S$, then $\xi=0$.
\end{lemma}
\begin{proof}
We prove $\xi(\un x)$ is identically zero by induction on the total length $L(\un x):=\sum_{i=1}^{n}|x_{i}|$.

If $L(\un x)<n$, then some $x_{i}=1$ and hence $\xi(\un x)=0$ since it is collapsing. If $L(\un x)=n$, then either some $|x_{i}|\ge2$, or all $x_{i}$ are simple reflections. In the latter case, if all $x_{i}$ are the same simple reflection, then $\xi(\un x)=0$ by \eqref{xis}; otherwise $\xi(\un x)=0$ since it is collapsing. 

Now assume $L(\un x)\ge n+1$, and assume  $\xi(\un x')=0$ whenever $L(\un x')<L(\un x)$. Let $\mu(\un x)=\min\{1\le i\le n; |x_{i}|\ge2\}$. We prove  $\xi(\un x)=0$ by downward induction on $\mu(\un x)$. 

If $\mu(\un x)=n$, then $x_{1},\cdots, x_{n-1}\in S\cup\{1\}$. Since $\xi$ is collapsing, $\xi(\un x)=0$ unless all $x_{i}$ are equal to the same $s\in S$ for $1\le i\le n-1$. In the case $x_{1}=\cdots=x_{n-1}=s$,  if $|sx_{n}|>|x_{n}|$, we still  have $\xi(\un x)=0$ by collapsing. If $|sx_{n}|<|x_{n}|$, we write $x_{n}=sx_{n}'$, then the cocycle condition $(\d\xi)(s,s,\cdots, s,s,x_{n}')=0$ together with  the collapsing condition implies $\xi(\un x)=\xi(s,\cdots, s)=0$ by \eqref{xis}.

Now assume $\xi(\un x')=0$ for those $\un x'$ such that either $L(\un x')<L(\un x)$ or $L(\un x')=L(\un x)$ and $\mu(\un x')>\mu(\un x)$. Let $\mu=\mu(\un x)$. By the same argument as in the previous paragraph, $\xi(\un x)=0$ unless $x_{1}=\cdots=x_{\mu-1}=s\in S$, and $x_{\mu}=sx'_{\mu}$ where $|x_{\mu}|>|x_{\mu'}|$. Now apply the cocycle condition $(\d\xi)(s,s,\cdots, s, x'_{\mu},x_{\mu+1},\cdots, x_{n})=0$ (the first $\mu$ entries are $s$). Expanding, we get the terms
\begin{eqnarray}
\label{L1}s\xi(s,s,\cdots, s, x'_{\mu},x_{\mu+1},\cdots, x_{n}), &\quad& \mbox{first $\mu-1$ entries are $s$}\\
\label{L2}\xi(s,\cdots,1,\cdots,s,x'_{\mu}, x_{\mu+1},\cdots, x_{n}), &\quad& \mbox{there are $\mu-1$ such terms}\\
\xi(s,\cdots, s, sx'_{\mu},\cdots, x_{n})=\xi(\un x),&&\\
\label{L4}\xi(s,\cdots, s, x'_{\mu}x_{\mu+1},x_{\mu+2},\cdots, x_{n}),& \quad &\mbox{first $\mu$ entries are $s$}\\
\label{L5}\xi(s,\cdots, s, x'_{\mu},\cdots, x_{i}x_{i+1},\cdots, x_{n}), &\quad& \mbox{first $\mu$ entries are $s$, $\mu<i<n$}\\
\label{L6}\xi(s,\cdots, s, x'_{\mu},x_{\mu+1},\cdots, x_{n-1}).&&
\end{eqnarray}
Now \eqref{L1} has total length less than $L(\un x)$ hence vanishes by inductive hypothesis; \eqref{L2} is zero by collapsing, \eqref{L4}, \eqref{L5} and \eqref{L6} have total length $\le L(\un x)$ and the $\mu$-invariant strictly larger than $\mu(\un x)$, hence they also vanish by inductive hypothesis. Since the sum of all terms is zero, the only remaining term $\xi(\un x)$ must also be zero. This finishes the induction step.
\end{proof}

\subsection{Proof of Theorem \ref{th:univ}}For $s\in S$  let $a_{s}:=\z(s,s,\cdots, s)\in A$.   Define a linear map $a:\ZZ[\cD]^{\ep(n)}\to A$ by sending $\bD_{H}:=(-1)^{[n/2]}[D^{+}_{H}]+(-1)^{[(n+1)/2]}[D^{-}_{H}]$ to $wa_{s}$ if $w\in W$ and $s\in S$ are such that $wD^{+}_{s}=D^{+}_{H}$ (such $(w,s)$ always exists).  To show $a$ is well-defined we need to show that if $(w',s')$ also satisfies $w'D^{+}_{s'}=D^{+}_{H}$, then $wa_{s}=w'a_{s'}$. Now $v=(w')^{-1}w$ satisfies $vD^{+}_{s}=D^{+}_{s'}$. By Lemma \ref{l:as}, $va_{s}=a_{s'}$, hence $wa_{s}=w'a_{s'}$. Therefore $a$ is well-defined. 

To see $a$ is  $W$-equivariant, we need to show that for any $H\in \cH$ and $w\in W$, $a(w\bD_{H})=wa(\bD_{H})$. Let $(w',s)\in W\times S$ be such that $w'D^{+}_{s}=D^{+}_{H}$. Then $a(\bD_{H})=w'a_{s}$ by definition. Now $wD^{+}_{H}$ is either $D^{+}_{wH}$ or $D^{-}_{wH}$. If $wD^{+}_{H}=D^{+}_{wH}$,  then $w\bD_{H}=\bD_{wH}$, hence $a(w\bD_{H})=a(\bD_{wH})$. Since $D^{+}_{wH}=ww'D^{+}_{s}$, we have $a(\bD_{wH})=ww'a_{s}$ by definition. Therefore $a(w\bD_{H})=a(\bD_{wH})=ww'a_{s}=wa(\bD_{H})$ holds. If $wD^{+}_{H}=D^{-}_{wH}$, then $w\bD_{H}=(-1)^{n}\bD_{wH}$, hence $a(w\bD_{H})=(-1)^{n}a(\bD_{wH})$. Since $D^{-}_{wH}=wD^{+}_{H}=ww'D^{+}_{s}$, we have $D^{+}_{wH}=ww'sD^{+}_{s}$, hence $a(\bD_{wH})=ww'sa_{s}$ by definition. By Lemma \ref{l:as}, $sa_{s}=(-1)^{n}a_{s}$, hence $a(\bD_{wH})=(-1)^{n}ww'a_{s}$. Combining these facts we get $a(w\bD_{H})=(-1)^{n}a(\bD_{wH})=ww'a_{s}=wa(\bD_{H})$. This verifies that  $a$ is $W$-equivariant. 

Now let $\xi=\z-a\c Z_{n}^{W}\in Z^{n}(W,A)$, which is again collapsing. From the construction of $a$ and Lemma \ref{l:two prop} we have
\begin{equation}\label{xis}
\xi(s,s\cdots, s)=a_{s}-a(\bD_{s})=0, \forall s\in S.
\end{equation}
By Lemma \ref{l:xi van}, $\xi$ is identically zero, hence $\z=a\c Z_{n}^{W}$.

Finally we prove the uniqueness of $a$. Suppose another $W$-equivariant map $a': \ZZ[\cD]^{\e(n)}\to A$ satisfies $\z=a'\c Z^{W}_{n}$, then $b=a-a'$ satisfies $b\c Z^{W}_{n}=0$. Now $b\c Z^{W}_{n}(s,s,\cdots ,s)=b(\bD_{s})=0$ for all $s\in S$. By $W$-equivariance, $b(\bD_{H})=0$ for all $H\in \cH$, hence $b=0$.  The proof is complete.

\begin{remark} The analogue of the collapsing condition for a $1$-cocycle $\xi\in Z^{1}(W,A)$ is the normalization $\xi(1)=0$. However, $Z^{W}_{1}$ is not in general the universal normalized $1$-cocycle for $W$.
\end{remark}

\begin{exam} When $n=2$, the universal collapsing 2-cocycle $Z^{W}_{2}$ defines an extension
\begin{equation}\label{ext W}
1\to \ZZ[\cH]\to W^{\#}\to W\to 1.
\end{equation}
Here $W^{\#}=\ZZ[\cH]\times W$ as a set, with multiplication $(a,x)(b,y)=(a+xb+Z^{W}_{2}(x,y), \wt x, \wt y)$, where $a,b\in \ZZ[\cH]=\ZZ[\cD]^{+}$, $x,y\in W$ and $\wt x,\wt y$ are the liftings of $x$ and $y$ to the positive braid monoid $B_{W}^{+}$ by reduced words.

When $W$ is a Weyl group, let $P_{W}\subset B_{W}$ be the pure braid group and the braid group associated to $W$. Let $\fra=\xcoch(T)_{\CC}$ and $\fra^{\rs}\subset \fra$ be the complement of the root hyperplanes. Then $P_{W}\cong \pi_{1}(\fra^{\rs}, *)$. Each root $\a$ gives a map $\fra^{\rs}\to \CC^{\times}$ hence a homomorphism $v_{\a}: P_{W}\to\pi_{1}(\CC^{\times})=\ZZ$. The homomorphism $v_{\a}$ only depends on the root hyperplane $H_{\a}$, hence they together define a homomorphism $v: P_{W}\to \ZZ[\cH]$. It is easy to see using the characterization of $Z^{W}_{2}$ that the extension \eqref{ext W} is the pushout of the extension $1\to P_{W}\to B_{W}\to W\to 1$ via the homomorphism $v$.
\end{exam}

\subsection{Cup product} Equip $\ZZ[\cD]$ with the (not necessarily unital) ring structure such that $[D]\cdot [D]=[D]$ for all $D\in \cD$ and $[D]\cdot [D']=0$ if $D\ne D'$ (i.e., the basis elements $[D]$ are orthogonal idempotents). Then $\ZZ[\cD]^{+}\oplus\ZZ[\cD]^{-}$ is a $\ZZ/2\ZZ$-graded subalgebra of $\ZZ[\cD]$. In particular, we have the multiplication map $\mu_{n}: (\ZZ[\cD]^{-})^{\ot n}\to \ZZ[\cD]^{\e(n)}$.

From the definition of $Z^{W}_{n}$, the following proposition is easily verified.
\begin{prop} The cocycle $Z^{W}_{n}$ is the image of the $n$th cup power of $Z^{W}_{1}$ under $\mu_{n}$:
\begin{equation*}
Z^{1}(W,\ZZ[\cD]^{-})^{\ot n}\xr{\cup} Z^{n}(W,(\ZZ[\cD]^{-})^{\ot n})\xr{\mu_{n}} Z^{n}(W,\ZZ[\cD]^{\e(n)}).
\end{equation*}
\end{prop}
\begin{proof} Let $Z'_{n}=\mu_{n}((Z_{1}^{W})^{\cup n})$. Then $Z_{n}'(x_{1},\cdots, x_{n})=Z_{1}^{W}(x_{1})\cdot x_{1}Z_{1}^{W}(x_{2})\cdots x_{1}\cdots x_{n-1}Z^{W}_{1}(x_{n})$ (where $\cdot$ is the multiplication in $\ZZ[\cD]$). For $H\in \cH$, $[D_{H}^{\pm}]$ appears in $\mu_{n}(Z_{1}^{W})^{\cup n}$ if and only if $[D^{\pm}_{H}]$ appears in each term $x_{1}\cdots x_{i-1}Z^{W}_{1}(x_{i})$, $1\le i\le n$, which happens if and only if $(C_{0},C_{1},\cdots, C_{n})$ is $H$-alternating. Moreover, $[D_{H}^{+}]$ appears in $x_{1}\cdots x_{i-1}Z^{W}_{1}(x_{i})$ with coefficient $(-1)^{i-1}$, therefore it appears in $Z_{n}'$ with coefficient $(-1)^{[n/2]}$. Comparing with \eqref{Zn'} we conclude that $Z_{n}'=Z^{W}_{n}$.
\end{proof}

\subsection{} We have  short exact sequences of $\ZZ[W]$-modules
\begin{eqnarray}
\label{+-}0\to \ZZ[\cD]^{+}\xr{i^{+}} \ZZ[\cD]\xr{p^{-}} \ZZ[\cD]^{-}\to 0\\
\label{-+}0\to \ZZ[\cD]^{-}\xr{i^{-}} \ZZ[\cD]\xr{p^{+}} \ZZ[\cD]^{+}\to 0
\end{eqnarray}
where $i^{\pm}$ are the natural inclusions, and $p^{\pm}([D])=[D]\pm[\s(D)]$ for all $D\in \cD$.

\begin{prop}\label{p:seesaw} For $n\ge1$, the cohomology class $[-Z^{W}_{n+1}]\in \cohog{n+1}{W, \ZZ[\cD]^{\e(n+1)}}$ is the image of $[Z^{W}_{n}]$ under the connecting homomorphism
\begin{equation*}
\cohog{n}{W, \ZZ[\cD]^{\e(n)}}\to \cohog{n+1}{W, \ZZ[\cD]^{\e(n+1)}}
\end{equation*}
attached to \eqref{+-} when $n$ is odd and to \eqref{-+} when $n$ is even. 
\end{prop}
\begin{proof}
Let $\wt Z_{n}$ be the lifting of $Z_{n}^{W}$ to a function $W^{n}\to \ZZ[\cD]$ defined by $\wt Z_{n}(x_{1},\cdots, x_{n})=\sum_{H}(-1)^{[n/2]}[D^{+}_{H}]$ where $H$ runs over those walls such that $(C_{0},C_{1},\cdots, C_{n})$ is $H$-alternating (as usual $C_{i}=x_{1}\cdots x_{i}C_{0}$). To prove the Proposition, it suffices to check that
\begin{equation*}
\d\wt Z_{n}=-Z_{n+1}^{W}.
\end{equation*} 
Expanding $\d\wt Z_{n}(x_{1},\cdots, x_{n+1})$ as we did for $\d Z^{W}_{n}$ in the proof of Prop.\ref{p:closed}. Note that all the terms except for the first one $x_{1}\wt Z_{n}(x_{2},\cdots, x_{n+1})$ only involve $[D_{H}^{+}]$ for various $H\in \cH$. Let $h_{i}=v_{D_{H}^{+}}(C_{i})$. The coefficient of $[D^{+}_{H}]$ in  $\d Z^{W}_{n}(x_{1},\cdots, x_{n+1})-x_{1}\wt Z_{n}(x_{2},\cdots, x_{n+1})$ is the sum of all but the first term in \eqref{hh}. Since the sum in \eqref{hh} is zero, the coefficient of $[D^{+}_{H}]$ in  $\d Z^{W}_{n}(x_{1},\cdots, x_{n+1})-x_{1}\wt Z_{n}(x_{2},\cdots, x_{n+1})$ is $-(h_{1}-h_{2})\cdots (h_{n}-h_{n+1})$, which is nonzero if and only if $(C_{1},\cdots, C_{n+1})$ is $H$-alternating, in which case the coefficient is $\pm1$. Since we know $\d Z_{n}=0$, $\d\wt Z_{n}$ must take values in $\ZZ[\cD]^{\e(n+1)}$, therefore, if $(C_{1},\cdots, C_{n+1})$ is $H$-alternating, and that $[D^{\pm}_{H}]$ both appear with nonzero coefficients in $\d\wt Z_{n}(x_{1},\cdots, x_{n+1})$, $[D^{-}_{H}]$ must appear with nonzero coefficient in $x_{1}\wt Z_{n}(x_{2},\cdots, x_{n+1})$. This means $x_{1}^{-1}D^{-}_{H}=D_{H'}^{+}$ for some $H'=x_{1}^{-1}H$, hence $C_{0}\subset x_{1}^{-1}D^{-}_{H}$, or $C_{1}\subset D^{-}_{H}$, i.e., $(C_{0},C_{1},\cdots, C_{n+1})$ is $H$-alternating. In this case, the coefficient of $[D_{H}^{+}]$ in $\d\wt Z_{n}(x_{1},\cdots, x_{n+1})$ is $-(h_{1}-h_{2})\cdots (h_{n}-h_{n+1})=-\prod_{i=0}^{n}(h_{i}-h_{i+1})$, which is the negative of the coefficient of $[D^{+}_{H}]$ in $Z_{n+1}^{W}$. This completes the proof.
\end{proof}

\section{Higher signs}\label{s:signs}

\begin{defn} 
\begin{enumerate}
\item For $n\ge2$ even, let  $\chi_{n}: \ZZ[\cD]^{+}=\ZZ[\cH]\to \ZZ$ be the linear map sending each basis element $[D^{+}_{H}]+[D^{-}_{H}]$ to $(-1)^{n/2}$ (where $H\in\cH$ ). 
\item For $n\ge1$ odd, let $\chi_{n}: \ZZ[\cD]^{-}\to \FF_{2}$ be the linear map sending each $[D^{+}_{H}]-[D^{-}_{H}]$ to $1\in \FF_{2}$. (Of course all the $\chi_{n}$ are the same for $n$ odd.)
\item For $n\ge1$, define $\e^{W}_{n}\in Z^{n}(W,\ZZ)$ for $n$ even and $\e^{W}_{n}\in Z^{n}(W,\FF_{2})$ for $n$ odd to be the image of $Z^{W}_{n}$ under the $W$-invariant map $\chi_{n}$.
\end{enumerate}
\end{defn}

Concretely, for $\un x=(x_{1},\cdots ,x_{n})\in W^{n}$, $\e_{n}(\un x)$ is the number (resp. the parity of the number when $n$ is odd) of walls $H\in \cH$ such that $(C_{0},C_{1},\cdots, C_{n})$ is alternating with respect to $H$. Here  $C_{i}=x_{1}\cdots x_{i}C$ for $1\le i\le n$.

\subsection{Proof of Theorem \ref{th:sign}}
Since $\e_{n}^{W}$ is the homomorphic image of a collapsing cocycle, it is also collapsing. By Lemma \ref{l:two prop}(2) and the definition of $\chi_{n}$,   we have $\e_{n}^{W}(s,\cdots, s)=1$. By Lemma \ref{l:xi van}, these two properties (collapsing and values at $(s,\cdots, s)$ ) also characterize $\e_{n}^{W}$ as $\ZZ$-valued or $\FF_{2}$-valued cocycles.

\begin{remark}[Lusztig] One can define $\e_{n}^{W}$ without using the notion of chambers and walls. Indeed, for  $w\in W$ there is a well-defined set of reflections $\bT_{w}$ in $W$ as in \cite[Ch. IV, \S1.4, Lemme 1,2]{B} (and $\# \bT_{w}=|w|$). For $(x_{1},\cdots ,x_{n})\in W^{n}$, let $\bT(x_{1},\cdots ,x_{n})$ be the intersection
\begin{equation*}
\bT(x_{1},\cdots ,x_{n}):=\bigcap_{i=1}^{n}x_{1}\cdots x_{i-1}\cdot \bT_{x_{i}}
\end{equation*}
where the $i=1$ term is $\bT_{x_{1}}$. Then
\begin{equation*}
\e_{n}^{W}(x_{1},\cdots ,x_{n})=\begin{cases} \# \bT(x_{1},\cdots ,x_{n}), & \mbox{ if $n$ is even};\\
\# \bT(x_{1},\cdots ,x_{n}) \mod 2, & \mbox{ if $n$ is odd}.
\end{cases}
\end{equation*}
To see this, we only need to observe that two chambers $xC_{0}$ and $xyC_{0}$ lie on different sides of the wall $H_{r}$ given by some reflection $r\in W$  if and only if $r\in x\bT_{y}$.
\end{remark}

\begin{remark} One can refine the cocycles $\e^{W}_{n}$ slightly as follows. Let $\un S$ be the set of $W$-orbits on $\cH$; the map $S\to \un S$ sending $s$ to the $W$-orbit of $H_{s}$ is surjective. For $n$ even, we have the $W$-invariant map $\wt\chi_{n}: \ZZ[\cD]^{+}\to \ZZ[\un S]$ sending $[D^{+}_{H}]+[D^{-}_{H}]$ to $ (-1)^{n/2}[\un H]$, where $\un H$ is the $W$-orbit of $H\in\cH$. For $n$ odd, we also have the $W$-invariant map $\wt\chi_{n}: \ZZ[\cD]^{-}\to \FF_{2}[\un S]$ sending $[D^{+}_{H}]-[D^{-}_{H}]$ to $[\un H]$. The image of $Z_{n}^{W}$ under $\wt\chi_{n}$ defines an $n$-cocycle $\wt\e_{n}^{W}\in Z^{n}(W, \ZZ[\un S])$ if $n$ is even and $\wt\e_{n}^{W}\in Z^{n}(W, \FF_{2}[\un S])$ if $n$ is odd. 
\end{remark}

\begin{prop} If $n\ge1$ is odd, then the cohomology class $[-\e^{W}_{n+1}]\in \cohog{n+1}{W,\ZZ}$ is the image of $[\e^{W}_{n}]$ under the Bockstein homomorphism
\begin{equation*}
\b: \cohog{n}{W,\FF_{2}}\to \cohog{n+1}{W,\ZZ}.
\end{equation*}  
\end{prop}
\begin{proof} Let $\th_{n+1}:\ZZ[\cD]\to \ZZ$ be the linear map sending all $[D]$ to $(-1)^{(n+1)/2}$. 
We have a commutative diagram of $\ZZ[W]$-modules where the rows are exact sequences 
\begin{equation*}
\xymatrix{0\ar[r] & \ZZ[\cD]^{+}\ar[r]^-{i^{+}}\ar[d]^{\chi_{n+1}} & \ZZ[\cD]\ar[d]^{\th_{n+1}} \ar[r]^-{p^{-}} & \ZZ[\cD]^{-} \ar[r]\ar[d]^{\chi_{n}} & 0\\
0\ar[r] & \ZZ\ar[r]^-{2} & \ZZ\ar[r] & \ZZ/2\ZZ\ar[r] & 0
}
\end{equation*}
This gives a commutative diagram of connecting homomorphisms
\begin{equation*}
\xymatrix{\cohog{n}{W,\ZZ[\cD]^{-}}\ar[r]\ar[d]^{\chi_{n}} & \cohog{n+1}{W,\ZZ[\cD]^{+}}\ar[d]^{\chi_{n+1}}\\
\cohog{n}{W,\FF_{2}}\ar[r]^{\b} & \cohog{n+1}{W,\ZZ}
}
\end{equation*}
The equality $[-\e^{W}_{n+1}]=\b[Z^{W}_{n}]$ then follows from Prop.\ref{p:seesaw}.
\end{proof}

The cocycles $\e_{1}^{W}$ and $\e_{2}^{W}$ are easy to compute.
\begin{prop}
\begin{enumerate}
\item The cocycle $\e^{W}_{1}$ is the sign homomorphism $W\ni x\mapsto |x|\mod2$.
\item For $x,y\in W$,  $\e^{W}_{2}(x,y)=\frac{1}{2}(|x|+|y|-|xy|)\in\ZZ$. 
\end{enumerate}
\end{prop}

The following properties of $\e^{W}_{n}$ follows directly from the uniqueness in Theorem \ref{th:sign}.
\begin{lemma}\label{l:prop sign}
\begin{enumerate}
\item $\e_{n}^{W}$ is invariant under any automorphism of $(W,S)$.
\item For any $x_{1},\cdots, x_{n}\in W$, we have $\e^{W}_{n}(x_{1},\cdots,x_{n})=\e^{W}_{n}(x_{n}^{-1}, \cdots ,x_{1}^{-1})$.
\item Suppose $(W',S')\subset (W,S)$ is a standard parabolic subgroup, then $\e^{W}_{n}|_{W'}=\e^{W'}_{n}$. 
\item Let $(W_{1},S_{1})$ and $(W_{2},S_{2})$ be Coxeter groups and $(W,S)=(W_{1}\times W_{2},S_{1}\coprod S_{2})$. Let  $p_{i}: W\to W_{i}$ be the projection for $i=1,2$. Then $\e_{n}^{W}=p_{1}^{*}\e_{n}^{W_{1}}+p_{2}^{*}\e_{n}^{W_{2}}$.
\end{enumerate}
\end{lemma}

\begin{lemma}\label{l:prod}
\begin{enumerate}
\item Suppose $W=\ZZ/2\ZZ$, then the cohomology class $[\e^{W}_{n}]$ is the nonzero element in $\cohog{n}{W,\ZZ}=\FF_{2}$ if $n$ is even or in $\cohog{n}{W,\FF_{2}}=\FF_{2}$ if $n$ is odd .
\item If $w\in W$ is conjugate to a product of $m$ {\em commuting} simple reflections (so $w$ is itself an involution), then the class $[\e_{n}|_{\j{w}}]=m\mod2$ as an element in  $\cohog{n}{\j{w}, \ZZ}=\FF_{2}$ if $n$  is even or  $\cohog{n}{W,\FF_{2}}=\FF_{2}$ if $n$ is odd.
\end{enumerate}
\end{lemma}

\begin{remark} The maps $\chi_{n}$ are not compatible with the ring structure on $\ZZ[\cD]$ in general (besides the case $W=\ZZ/2\ZZ$). Therefore, in general there is no reason to expect that $\e^{W}_{n}$ be related to cup powers of the sign $\e^{W}_{1}$ in any obvious way.
\end{remark}

In the rest of the section we collect some useful formulas for computing $\e^{W}_{n}$ and more generally collapsing cocycles.

\begin{lemma}\label{l:odd1} Let $x_{1},\cdots ,x_{n}\in W$ and let $\ell\ge1$ be odd. Suppose there is a partial product of $\ell$ consecutive entries $x_{j}\cdots x_{j+\ell-1}=1$,  then $\z(x_{1},\cdots,x_{n})=0$ for any collapsing $n$-cocycle $\z\in Z^{n}(W,A)$.
\end{lemma}
\begin{proof} By Theorem \ref{th:univ} it suffices to treat the case $\z=Z^{W}_{n}$.  
If $x_{j}\cdots x_{j+\ell-1}=1$, consider the sequence of $\ell+1$ chambers $(C_{j-1},C_{j}, \cdots ,C_{j+\ell-1})$. The condition $x_{j}\cdots x_{j+\ell-1}=1$ implies $C_{j+\ell-1}=C_{j-1}$. For any $H\in \cH$, this even number of chambers cannot be $H$-alternating if the first one is the same as the last one. Therefore $Z_{n}^{W}(x_{1},\cdots ,x_{n})=0$ by \eqref{Zn'}.
\end{proof}

\begin{lemma}\label{l:even1}
Let $x_{1},\cdots ,x_{n}\in W$ and let $\ell\ge2$ be even. If $x_{1}\cdots x_{\ell}=1$, then
\begin{equation}\label{inv1}
\e^{W}_{n}(x_{1},x_{2},x_{3},\cdots, x_{n})\equiv\e^{W}_{n-1}(x_{2},x_{3},\cdots, x_{n})\mod2.
\end{equation}
Similarly, if  $x_{n-\ell+1}\cdots x_{n}=1$, then 
\begin{equation}\label{inv2}
\e^{W}_{n}(x_{1},x_{2},\cdots, x_{n-1},x_{n})\equiv\e^{W}_{n-1}(x_{1},x_{2},\cdots, x_{n-1})\mod 2.
\end{equation}
\end{lemma}
\begin{proof}
We prove \eqref{inv1}. Since $x_{1}\cdots x_{\ell}=1$, the corresponding chain of chambers $(C_{0},C_{1},\cdots, C_{n})$ satisfies $C_{\ell}=C_{0}$. For any $H\in \cH$, $(C_{0},C_{1},\cdots, C_{n})$  is $H$-alternating if and only if $(C_{1},\cdots ,C_{n})$ is $H$-alternating, if and only if $(C_{0}, x_{2}C_{0},\cdots, x_{2}\cdots x_{n}C_{0})$ is $x_{2}H$-alternating. This implies \eqref{inv1}. The proof of \eqref{inv2} is similar.
\end{proof}

\begin{lemma}\label{l:exp} Let $x_{1},\cdots, x_{n}\in W$ and $1\le i\le n$. Let  $x_{i}=s_{1}\cdots s_{m}$ be a reduced expression. Then for any collapsing $n$-cocycle $\z\in Z^{n}(W,A)$, we  have
\begin{equation*}
\z(x_{1},\cdots,x_{n})=\sum_{j=1}^{m}\z(x_{1},\cdots, x_{i-1}s_{1}\cdots s_{j-1}, s_{j}, s_{j+1}\cdots s_{m}x_{i+1}, \cdots ,x_{n}).
\end{equation*} 
Here, if $i=1$, the first term should be $s_{j}$; if $i=n$, the last term should be $s_{j}$.
\end{lemma}
\begin{proof}
We prove this by induction on $m$. For $m=0$ the statement is vacuous. Now assume the statement is proved for $m-1$.  Applying the cocycle condition $(\d\z)(x_{1},\cdots,x_{i-1}, s_{1},s_{2}\cdots s_{m}, x_{i+1},\cdots,x_{n})=0$ we get a sum of terms most of which contain two consecutive entries $s_{1}$ and $s_{2}\cdots s_{m}$, which vanishes since $\z$ is collapsing. The remaining terms are 
\begin{eqnarray*}
\z(x_{1},\cdots, x_{i-1} s_{1},s_{2}\cdots s_{m}, x_{i+1},\cdots, x_{n}) && \mbox{(if $i=1$ the first entry is $s_{2}\cdots s_{m}$)};\\
\z(x_{1},\cdots, x_{i-1}, x_{i}, x_{i+1},\cdots, x_{n});&&\\
\z(x_{1},\cdots, x_{i-1}, s_{1}, s_{2}\cdots s_{m}x_{i+1},x_{i+2},\cdots, x_{n})& &\quad \mbox{(if $i=n$, the last entry is $s_{1}$)}.
\end{eqnarray*}
The sum of the three terms above are zero. We then apply inductive hypothesis to the first term above to conclude.
\end{proof}

\section{Geometric origin of $\e^{W}_{3}$}\label{s:G}
When $W$ is a Weyl group, the cocycle $\e_{3}^{W}$ naturally shows up in the study of convolutions in certain Hecke categories. In fact this was our motivation to study collapsing cocycles for general Coxeter groups. Below we recall the geometric context.
   
\subsection{Monodromic Hecke category} In  \cite{LY} we study monodromic Hecke categories for reductive groups. Let $G$ be a connected reductive group over an algebraically closed field $k$, $B$ a Borel subgroup of $G$ with unipotent radical $U$ and $T\subset B$ a maximal torus.  Let  $\cL$ and $\cL'$  be two rank one character sheaves on $T$. In \cite{LY} we study the derived category ${}_{\cL'}\cD_{\cL}$ of complexes on $U\bs G/U$ equivariant under the left and right translation by $T$ with respect to $\cL$ and $\cL'$ respectively. These categories appear naturally in the study of representations of finite groups of Lie type. 

Now fix a $W$-orbit $\fo$ of rank one character sheaves on $T$. The direct sum $\cD_{\fo}:=\oplus_{\cL,\cL'\in\fo}({}_{\cL'}\cD_{\cL})$ carries a monoidal structure under convolution
\begin{equation}
\star: {}_{\cL''}\cD_{\cL'}\times{}_{\cL'}\cD_{\cL}\to {}_{\cL''}\cD_{\cL}.
\end{equation}
In \cite[\S4]{LY}, we gave a decomposition of $\cD_{\fo}$ into blocks. For $\cL,\cL'\in \fo$, let ${}_{\cL'}W_{\cL}=\{w\in W|w\cL=\cL'\}$. Let $W_{\cL}^{\c}$ be the subgroup of $W$ generated by reflections with respect to the roots $\a$ such that $\a^{\vee,*}\cL$ is the trivial local system on $\Gm$. Then the blocks in ${}_{\cL'}\cD_{\cL}$ are indexed by the set ${}_{\cL'}\un W_{\cL}={}_{\cL'}W_{\cL}/W^{\c}_{\cL}=W^{\c}_{\cL'}\bs {}_{\cL'}W_{\cL}$ (see \cite[\S4.1]{LY}). Moreover, each coset $\b\in {}_{\cL'}\un W_{\cL}$ contains a unique element $w^{\b}\in {}_{\cL'}\un W_{\cL}$ of minimal length (\cite[Lemma 4.2]{LY}).   

In \cite[\S4.4]{LY} we introduced a groupoid $\Xi$ whose objects set is $\fo$ and its morphism set $\Hom_{\Xi}(\cL,\cL')$ is ${}_{\cL'}\un W_{\cL}$, the set of blocks in ${}_{\cL'}\cD_{\cL}$. The multiplication on the groupoid $\Xi$ is compatible with convolution on $\cD_{\fo}$. In \cite[\S5.8]{LY} we have defined a $3$-cocycle $\s$ on the groupoid $\Xi$ that captures certain signs that appear in the associativity constraint of the convolution, which we now recall.

Let $\cL,\cL',\cL''\in \fo$, and let $\b\in {}_{\cL'}\un W_{\cL}$ and $\g\in{}_{\cL''}\un W_{\cL'}$ be blocks. Let $\dw^{\b},\dw^{\g}$ be liftings of $w^{\b},w^{\g}$ to $N_{G}(T)$. Let $\D(\dw^{\b})_{\cL}$ be the standard perverse sheaf obtained from a shifted local system on $B\dw^{\b}B$ by extension by zero (equipped with a trivialization of its stalk at $\dw^{\b}$, see \cite[\S2.11]{LY}).  By \cite[\S5.5]{LY}, there is a {\em canonical} isomorphism in ${}_{\cL''}\cD_{\cL}$
\begin{equation*}
\can_{\dw^{\g},\dw^{\b}}:\D(\dw^{\g})_{\cL'}\star \D(\dw^{\b})_{\cL}\cong \D(\dw^{\g}\dw^{\b})_{\cL}.
\end{equation*}
For another $\cL'''\in\fo$, and block $\d\in {}_{\cL'''}W_{\cL''}$ consider the following diagram of isomorphisms
\begin{equation}\label{noncomm}
\xymatrix{ & \D(\dw^{\d}\dw^{\g})_{\cL'}\star \D(\dw^{\b})_{\cL}\ar[dr]^-{\can_{\dw^{\d}\dw^{\g},\dw^{\b}}} & \\
\D(\dw^{\d})_{\cL''}\star\D(\dw^{\g})_{\cL'}\star\D(\dw^{\b})_{\cL}\ar[ur]^-{\can_{\dw^{\d},\dw^{\g}}\star\id}\ar[dr]_-{\id\star\can_{\dw^{\g},\dw^{\b}}} & & \D(\dw^{\d}\dw^{\g}\dw^{\b})_{\cL}\\
& \D(\dw^{\d})_{\cL''}\star \D(\dw^{\g}\dw^{\b})_{\cL}\ar[ur]_-{\can_{\dw^{\d},\dw^{\g}\dw^{\b}}} &}
\end{equation}
However, this diagram is not necessarily commutative. Let $\s(\dw^{\d},\dw^{\g},\dw^{\b})_{\cL}\in \Qlbar^{\times}$ be the ratio of the upper composition over the lower. It is easy to see that $\s(\dw^{\d},\dw^{\g},\dw^{\b})_{\cL}$ is independent of the choices of liftings of $w^{\b},w^{\g}$ and $w^{\d}$; therefore we may denote it by $\s(\d,\g,\b)_{\cL}$. From the pentagon axiom for the associativity of the convolution functor,  we get that $\s$ is a $3$-cocycle on the groupoid $\Xi$, i.e., $\s\in Z^{3}(\Xi,\Qlbar^{\times})$.

\subsection{} Our next goal is to relate the cocycle $\s\in Z^{3}(\Xi,\Qlbar^{\times})$ to the cocycle $\e^{W}_{3}\in Z^{3}(W,\FF_{2})$. By definition, there is a map of groupoids
\begin{equation}
\pi: \Xi\to [\pt/W]
\end{equation}
sending a  morphism $\b\in \Hom_{\Xi}(\cL,\cL')={}_{\cL'}\un W_{\cL}$ to $w^{\b}\in W$.

We  generalize the notion of a collapsing cocycle to the groupoid $\Xi$. Let $n\ge2$ and $A$ be an abelian group. A cocycle $\z\in Z^{n}(\Xi, A)$ is collapsing if, whenever $\b_{1},\cdots, \b_{n}$ are composable morphisms in $\Xi$ (so that $\b_{1}\cdots\b_{n}$ is defined), and $|w^{\b_{i}}w^{\b_{i+1}}|=|w^{\b_{i}}|+|w^{\b_{i+1}}|$ for some $1\le i\le n-1$, we have $\z(\b_{1},\cdots, \b_{n})=0$. We have the following generalization of Lemma \ref{l:xi van}.

\begin{lemma}\label{l:var xi van} Let $n\ge2$ and $\z\in Z^{n}(\Xi,A)$ be a collapsing cocycle. Suppose for any morphism $\b$ in $\Xi$ containing a simple reflection $s\in S$, we have $\z(\b^{(-1)^{n-1}},\cdots, \b^{-1}, \b)=0$.  Then $\z=0$.
\end{lemma}
\begin{proof}[Proof sketch] The proof is almost identical as that of Lemma \ref{l:xi van}. For composable morphisms $\un\b=(\b_{1},\cdots, \b_{n})$ in $\Xi$, we prove the vanishing of $\z(\b_{1},\cdots, \b_{n})$ by induction on $L(\un \b):=\sum_{i=1}^{n}|w^{\b_{i}}|$. The only nontrivial point is to make sense of reduced word decomposition of a morphism in $\Xi$. Let $\b\in \Hom_{\Xi}(\cL,\cL')$, and  $w^{\b}=s_{i_{N}}s_{i_{n-1}}\cdots s_{i_{1}}$  a reduced expression in $W$. Let $\cL_{j}=s_{i_{j}}\cdots s_{i_{1}}\cL$. Then for any $1\le j< j'\le N$, the partial product $s_{i_{j'}}\cdots s_{i_{j}}$ is the minimal length element in its coset $\b_{[j,j']}\in {}_{\cL_{j'}}\un W_{\cL_{j}}$, i.e., $w^{\b_{[j,j']}}=s_{i_{j'}}\cdots s_{i_{j}}$. This follows from \cite[Lemma 4.5(3)(4)]{LY}.  In particular, we may write $\b=\b_{[N,N]}\circ \b_{[1,N-1]}$ and $w^{\b}=s_{i_{N}}w^{\b_{[1,N-1]}}$ so that $|w^{\b_{N}}|=N=1+|w^{\b_{[1,N-1]}}|$. This decomposition is used in the induction step. 
\end{proof}

\begin{theorem} The cocycle $\s$ takes values in $\{\pm1\}$. If we identify $\{\pm1\}$ with $\FF_{2}$, we have $\s=\pi^{*}\e^{W}_{3}$. 
\end{theorem}
\begin{proof} We first check that $\s$ is collapsing. For this we look more closely into the geometry behind the cocycle $\s$. Let $X$ be the flag variety of $G$, and $\bO(w)\subset X\times X$ be the $G$-orbit containing $(B,\Ad(w)B)$.  For any sequence of elements $w_{n},\cdots ,w_{1}\in W$, let $X(w_{n},\cdots ,w_{1})\subset X^{n}$ be the variety classifying sequences of Borel subgroups $(B_{1},\cdots ,B_{n})$ such that $(B, B_{1})\in \bO(w_{1})$, and $(B_{i},B_{i+1})\in \bO(w_{i+1})$ for $1\le i\le n-1$. For composable morphisms $\d,\g$ and $\b$ in $\Xi$, let $X(w^{\d}\#w^{\g},w^{\b})\subset X(w^{\d},w^{\g},w^{\b})$ be the closed subvariety where $(B_{1},B_{3})\in\bO(w^{\d}w^{\g})$; let $X(w^{\d},w^{\g}\#w^{\b})\subset X(w^{\d},w^{\g},w^{\b})$ be the closed subvariety where $(B,B_{2})\in\bO(w^{\g}w^{\b})$. We have a diagram
\begin{equation*}
\xymatrix{ & X(w^{\d}\#w^{\g},w^{\b}) \ar@{_{(}->}[dl]_{i_{1}}\ar[r]^{p_{1}} & X(w^{\d}w^{\g},w^{\b})\ar[dr]^{q_{1}}\\
X(w^{\d},w^{\g},w^{\b}) & & & X\\
& X(w^{\d},w^{\g}\#w^{\b}) \ar@{_{(}->}[ul]_{i_{2}}\ar[r]^{p_{2}} & X(w^{\d},w^{\g}w^{\b})\ar[ur]^{q_{2}}
}
\end{equation*}
The maps are the obvious inclusion and forgetful maps.  We take the fiber of all spaces in the above diagram over the point $B_{3}=\Ad(w^{\d}w^{\g}w^{\b})B\in X$, and name them $Y(?)$ instead of $X(?)$. Then the stalk of $\D(\dw^{\d})_{\cL''}\star\D(\dw^{\g})_{\cL'}\star\D(\dw)_{\cL}$ at $\dw^{\d}\dw^{\g}\dw^{\b}$ is $\cohoc{*}{Y(w^{\d},w^{\g},w^{\b}),\cF[|w^{\d}|+|w^{\g}|+|w^{\b}|]}$ for some rank one local system $\cF$ constructed from $\cL'',\cL'$ and $\cL$. The restriction of $\cF$ to $Y(w^{\d}\# w^{\g},w^{\b})$ and $Y(w^{\d},w^{\g}\#w^{\b})$ are canonically isomorphic to $\Qlbar$. By \cite[\S5.3]{LY}, the maps $q_{1}p_{1}$ and $q_{2}p_{2}$ are affine space bundles of dimension $d=\frac{1}{2}(|w^{\d}|+|w^{\g}|+|w^{\b}|-|w^{\d}w^{\g}w^{\b}|)$ when restricted to the point $\Ad(w^{\d}w^{\g}w^{\b})B$. Therefore $Y(w^{\d}\#w^{\g},w^{\b})$ and $Y(w^{\d},w^{\g}\#w^{\b})$ are isomorphic to $\AA^{d}$. The fact that $\can_{\dw^{\d},\dw^{\g}}$ and $\can_{\dw^{\g},\dw}$ are isomorphisms implies that the restriction maps 
\begin{eqnarray*}
i^{*}_{1}: \cohoc{-|w^{\d}w^{\g}w^{\b}|}{Y(w^{\d},w^{\g},w^{\b}),\cF[|w^{\d}|+|w^{\g}|+|w^{\b}|]}\to \cohoc{2d}{Y(w^{\d}\# w^{\g},w^{\b}), \Qlbar}\cong \Qlbar\\
i^{*}_{2}:\cohoc{-|w^{\d}w^{\g}w^{\b}|}{Y(w^{\d},w^{\g},w^{\b}),\cF[|w^{\d}|+|w^{\g}|+|w^{\b}|]}\to \cohoc{2d}{Y(w^{\d}, w^{\g}\#w^{\b}), \Qlbar}\cong \Qlbar
\end{eqnarray*}
are isomorphisms. The number $\s(\d,\g,\b)$ is the ratio of these two isomorphisms. 

Suppose $|w^{\d}|+|w^{\g}|=|w^{\d}w^{\g}|$, then $Y(w^{\d}\# w^{\g},w^{\b})=Y(w^{\d},w^{\g},w^{\b})$, therefore $Y(w^{\d}, w^{\g}\#w^{\b})=Y(w^{\d},w^{\g},w^{\b})$ as well for dimension reasons. In this case, both $i^{*}_{1}$ and $i^{*}_{2}$ are the identity maps, hence $\s(\d,\g,\b)=1$. The case $|w^{\g}|+|w^{\b}|=|w^{\g}w^{\b}|$ is proved similarly.  This proves that $\s$ is collapsing.

In \cite[Example 5.7]{LY}, we have shown for $G=\SL_{2}$, and a block $\b$ containing the nontrivial element in the Weyl group $W=S_{2}$, we have  $\s(\b,\b^{-1},\b)=-1$. This implies for general $G$, and any block $\b$ containing a simple reflection $s\in S$,  we have $\s(\b,\b^{-1},\b)=-1$ since the calculation in this case reduces to the $\SL_{2}$ example by the homomorphism $\SL_{2}\to G$ corresponding to $s$. Now we identify $\{\pm1\}$ with $\FF_{2}$. We have  seen that $\z=\s-\pi^{*}\e^{W}_{3}$  is collapsing in $Z^{3}(\Xi, \FF_{2})$ and satisfies $\z(\b,\b^{-1},\b)=0$ for any morphism $\b$ in $\Xi$ containing a simple reflection. By Lemma \ref{l:var xi van}, $\z$ vanishes identically, hence  $\s=\pi^{*}\e^{W}_{3}$.
\end{proof}

\begin{remark} A similar sign is observed by Bezrukavnikov and Riche \footnote{Private communication.} in the usual Hecke category. Let $\D(s)$ and $\nb(s)$ be the $!$ and $*$ extensions of $\Qlbar[1]$ from the Schubert cell $X(s)\subset X=G/B$.  Let $\d$ be the skyscraper sheaf at the point stratum $X(1)=\{B\}$. There are canonical isomorphisms $\D(s)\star\nb(s)\cong\d\cong\nb(s)\star\D(s)$. However, the resulting two isomorphisms
\begin{eqnarray}
\label{D1} (\D(s)\star \nb(s))\star \D(s)\cong \d\star\D(s)=\D(s),\\
\label{D2} \D(s)\star (\nb(s)\star \D(s))\cong \D(s)\star\d=\D(s)
\end{eqnarray} 
differ by a sign. Indeed, the stalk of $\D(s)\star \nb(s)\star \D(s)$ along $X(s)$ can be identified with
\begin{equation*}
\cohoc{*}{\AA^{2}, j_{*}\Qlbar[3]}
\end{equation*}
where $j:\AA^{2}-C\incl \AA^{2}$, and $C=\{(x,y)\in\AA^{2}|xy=1\}$. The two isomorphisms \eqref{D1} and \eqref{D2} correspond to the two restrictions which are isomorphisms
\begin{eqnarray*}
\res_{x}: \cohoc{*}{\AA^{2}, j_{*}\Qlbar[3]}\isom \cohoc{*}{\AA^{1}_{x}, \Qlbar[3]}\cong \Qlbar[1](-1),\\
\res_{y}: \cohoc{*}{\AA^{2}, j_{*}\Qlbar[3]}\isom \cohoc{*}{\AA^{1}_{x}, \Qlbar[3]}\cong \Qlbar[1](-1).
\end{eqnarray*}
Here $\AA^{1}_{x}$ and $\AA^{1}_{y}$ are the coordinate lines $y=0$ and $x=0$. The involution $\s:(x,y)\mapsto (y,x)$ then interchanges $\res_{x}$ and $\res_{y}$. To see the sign, it suffices to show that $\s^{*}$ acts on $\cohoc{2}{\AA^{2}, j_{*}\Qlbar}$ by $-1$. Using the triangle $\Qlbar\to j_{*}\Qlbar\to i_{*}\Qlbar[-1](-1)$, where $i:C\incl \AA^{2}$, we get $\cohoc{2}{\AA^{2}, j_{*}\Qlbar}\cong \cohoc{1}{C,\Qlbar}(-1)$. Now $C\cong\Gm$ and $\s$ restricts to the inversion map on $C$, which acts by $-1$ on $\cohog{1}{C}$ and $\cohoc{1}{C}$. This is the source of the sign.
\end{remark}

\section{Restriction of $\e^{W}_{3}$ to $\Om$} In this section, for $W$ an irreducible Weyl group,  we calculate the restriction of $\e^{W}_{3}$ to certain abelian subgroups. Such restriction naturally shows up in \cite{LY} as the twisting data for the monoidal structure of monodromic Hecke categories, and our calculation here shows that the twisting is often nontrivial.

Below we write $\e^{W}_{3}$ simply as $\e_{3}$.  We first collect some useful properties of $\e_{3}$.
\begin{prop}\label{p:z3} Let $(W,S)$ be any Coxeter group and  $x,y,z\in W$.
\begin{enumerate}
\item If $y=s\in S$, then
\begin{equation*}
\e_{3}(x,s,z)=\begin{cases} 1 & \mbox{if $|xs|<|x|$ and $|sz|<|z|$}; \\ 0 & \mbox{otherwise.}\end{cases}
\end{equation*}
\item Let $y=s_{1}\cdots s_{m}$ be a reduced expression, then 
\begin{equation*}
\e_{3}(x,y,z)=\sum_{i=1}^{m}\e_{3}(xs_{1}\cdots s_{i-1}, s_{i}, s_{i+1}\cdots s_{m}z).
\end{equation*}
\item We have $\e_{3}(x,y,z)=0$ if $xyz=1$.
\item We have $\e_{3}(y^{-1},y,z)=\frac{1}{2}(|y|+|z|-|yz|)\mod2$ and $\e_{3}(x,y,y^{-1})=\frac{1}{2}(|x|+|y|-|xy|)\mod2$.
\item We have $\e_{3}(x,x^{-1},x)=|x|\mod2$.
\item We have $\e_{3}(x,y,z)=\e_{3}(y,z,(xyz)^{-1})=\e_{3}(z,(xyz)^{-1},x)=\e_{3}((xyz)^{-1},x,y)$.
\end{enumerate}
\end{prop}
\begin{proof}
(1) follows from definition. (2) is a special case of Lemma \ref{l:exp}.



(3) follows from Lemma \ref{l:odd1}. (4) follows from Lemma \ref{l:even1} and (5) is a special case of (4). 

(6) Apply the cocycle condition $(\d\e_{3})(x,y,z,(xyz)^{-1})=0$, and using (3) to get the first equality; the others are obtained by iterating the first one.

\end{proof}

\subsection{The group $\Om$} For the rest of the notes, let $W$ be a Weyl group.  Let $G,T,B$ be as in \S\ref{ss:not}, with the extra assumption that $G$ is simple of adjoint type. Let $\tilW=\L\rtimes W$ be the extended affine Weyl group attached to $G$, where $\L=\xcoch(T)$. Let $A$ be the corresponding fundamental alcove in $\xcoch(T)_{\RR}$, and let $\Om\subset \tilW$ be the stabilizer of $A$. The map $i: \Om\subset \tilW\to W$ is injective, and the conjugacy class of $i$ is independent of the choice of $B$ (it will depend on $G$ and not just $W$, as we shall see by comparing type $B$ and type $C$).  

The finite abelian group $\Om_{\cL}$ in \cite{LY} is a subgroup of $\Om$. We would like to compute the cohomology classes of the pullback of $\e_{3}$ to $\Om$ as well as to subgroups of $\Om$, which play a role in the statement of \cite[Theorem 10.12]{LY}. This question is non-vacuous only when $\#\Om$ is even, which happens for types $A_{n}$ ($n$ odd), $B_{n}$, $C_{n}$, $D_{n}$ and $E_{7}$.

\subsection{Method} To compute $[\e_{3}|_{\Om}]$ we need to compute $i$. Let $D$ be the Dynkin diagram of $G$, and $\wt D$ be the extended Dynkin diagram of $G$.  It is well-known that $\Om\cong \Aut(\wt D)/\Aut(D)$. Also $\Aut(D)=\Aut(\wt D,\a_{0})$. We shall compute the projection $\wt i: \Aut(\wt D)\to W$.  Let $\t\in \Aut(\wt D)$, then it is a permutation of affine simple roots. It is then easy to write down the matrix for the action of $\t$ on $V=\xch(T)_{\QQ}=\Span_{\QQ}\{\a_{s}|s\in S\}$. If $\t(\a_{s})=\a_{0}$, then its action on $V$ should read $\t(\a_{s})=-\th=-\sum_{s}n_{s}\a_{s}$.  

Another general principle: if $X\subset W$, let $\Phi(X)=\{\a\in \Phi(G,T)|\a\bot V^{X}\}$, and let $W(X)\subset W$ be the subgroup generated by reflections attached to $\a\in \Phi(X)$. Then $X\subset W(X)$ and $W(X)$ is conjugate to a standard parabolic subgroup of $W$.   This allows us to reduce the calculation to smaller Coxeter groups.

In the next subsections we compute $[\e_{3}|_{\Om}]$ case by case.
\subsection{Type $A_{n-1}$, $n$ even}\label{ss:A}
We have $\Om\cong\ZZ/n\ZZ\subset S_{n}$ generated by the cyclic permutation $C=s_{1}s_{2}\cdots s_{n-1}$.  We claim that $[\e_{3}|_{\Om}]$ is the nontrivial class in $\cohog{3}{\Om,\FF_{2}}=\FF_{2}$. Let $n=2m$.

Suppose in the contrary that $\e_{3}|_{\Om}$ is a cobounbary, i.e., $\e_{3}=\d\l$ for some function $\l: \Om\times\Om\to \FF_{2}$, then the fact that $\e_{3}$ is normalized shows that $\l(x,1)=\l(1,1)=\l(1,x)$ for all $x\in \Om$. We may assume $\l(1,1)=0$ for otherwise we can change $\l$ to $\l+1$ and its coboundary is still $\e_{3}$. Therefore we assume $\l$ is also normalized, i.e., $\l(1,x)=\l(x,1)=0$ for all $x\in \Om$. 

We have $\e_{3}(C^{m},C^{m},C)=\e_{3}(C,C^{m},C^{m})$ by Prop. \ref{p:z3}(4).  Since $\e_{3}=\d\l$, we get
\begin{eqnarray*}
\e_{3}(C^{m},C^{m},C)=\l(C^{m},C)+\l(C^{m},C^{m+1})+\l(C^{m},C^{m})\\
\e_{3}(C,C^{m},C^{m})=\l(C,C^{m})+\l(C^{m+1},C^{m})+\l(C^{m},C^{m}).
\end{eqnarray*}
The equality of the LHS of the above two equations implies
\begin{equation}\label{4 lam}
\l(C^{m},C)+\l(C,C^{m})+\l(C^{m},C^{m+1})+\l(C^{m+1},C^{m})=0.
\end{equation}
However, the above is equal to $\e_{3}(C^{m},C,C^{m})$, which we now show is nonzero. 

The element $C$ is the cyclic permutation: $1\mapsto 2\mapsto \cdots\mapsto n\mapsto 1$. The involution $C^{m}$ interchanges $i$ and $i+m$ for $1\le i\le m$, and hence $|C^{m}|=m^{2}$. The permutation $C^{m}s_{1}\cdots s_{m-1}$ is $j\mapsto m+j+1$ for $1\le j\le m-1$, $m\mapsto m+1$, $m+j\mapsto  j$ for $0<j\le m$. We see that $|C^{m}s_{1}\cdots s_{m-1}|=m^{2}+m-1=|C^{m}|+m-1$. Similarly, $|s_{m+1}\cdots s_{n-1}C^{m}|=m-1+|C^{m}|$. By Prop. \ref{p:z3}(1)(2), we see
\begin{equation}\label{z3Cm}
\e_{3}(C^{m},C,C^{m})=\sum_{j=1}^{n-1}\e_{3}(C^{m}s_{1}\cdots s_{j-1}, s_{j}, s_{j+1}\cdots s_{n-1}C^{m}).
\end{equation}
The automorphism of $(W,S)$ sending $s_{i}$ to $s_{n-i}$ sends the triple $(C^{m}s_{1}\cdots s_{j-1}, s_{j}, s_{j+1}\cdots s_{n-1}C^{m})$ to $(C^{m}s_{1}\cdots s_{n-j-1}, s_{n-j}, s_{n-j+1}\cdots s_{n-1}C^{m})$, which have the same value under $\e_{3}$ by Lemma \ref{l:prop sign}(1). Therefore the RHS of \eqref{z3Cm} reduced to one term $\e_{3}(C^{m}s_{1}\cdots s_{m-1}, s_{m}, s_{m+1}\cdots s_{n-1}C^{m})$. Now $|C^{m}s_{1}\cdots s_{m}|=m^{2}+m-2=|C^{m}s_{1}\cdots s_{m-1}|-1$, and similarly $|s_{m}\cdots s_{n-1}C^{m}|=|s_{m+1}\cdots s_{n-1}C^{m}|-1$. By Prop. \ref{p:z3}(1),   $\e_{3}(C^{m}s_{1}\cdots s_{m-1}, s_{m}, s_{m+1}\cdots s_{n-1}C^{m})=1$, therefore $\e_{3}(C^{m},C,C^{m})\ne0$, which contradicts \eqref{4 lam}. Therefore $[\e_{3}|_{\Om}]$ is the nontrivial element in $\cohog{3}{\Om,\FF_{2}}$.

For any subgroup $\Om'\subset \Om$, the restriction $[\e_{3}|_{\Om'}]\in\cohog{3}{\Om',\FF_{2}}$ is nontrivial if and only if $[\Om:\Om']$ is odd.


\subsection{Type $B_{n}$}
In this case, $\Om=\ZZ/2\ZZ$, with nontrivial element conjugate to the simple reflection $s_{n}: (x_{1},\cdots, x_{n})\mapsto (x_{1},x_{2}\cdots, -x_{n})$. Therefore $[\e_{3}|_{\Om}]\in \cohog{3}{\Om,\FF_{2}}=\FF_{2}$ is always the nontrivial element.

\subsection{Type $C_{n}$}
In this case, $\Om=\ZZ/2\ZZ$, the nontrivial element is the involution $\t: (x_{1},\cdots, x_{n})\mapsto (-x_{n},\cdots, -x_{1})$. We have $\Phi(\t)=\{\e_{1}+\e_{n}, \e_{2}+\e_{n-1},\cdots\}$ (if $n$ is odd, $2\e_{(n+1)/2}\in \Phi(\t)$) consisting of orthogonal roots. Therefore $W(\t)$ is of type $A_{1}^{[(n+1)/2]}$, and $\t$ is the product of nontrivial elements from each factor. By Lemma \ref{l:prod}(2), $[\e_{3}|_{\Om}]=[\frac{n+1}{2}]\mod 2\in \cohog{3}{\Om,\FF_{2}}=\FF_{2}$. Therefore, 
\begin{equation*}
[\e_{3}|_{\Om}]\begin{cases} =0 & n\equiv0,3\mod4 \\ \ne0 & n\equiv1,2\mod4.\end{cases}
\end{equation*}

\subsection{Type $D_{n}$, $n$ odd}
In this case $\Om\cong \ZZ/4\ZZ$. The three nontrivial elements in $i(\Om)\subset W$ are $\om,\om^{2}, \om^{3}$:
\begin{eqnarray*}
\om(x_{1},\cdots, x_{n})=(x_{n},-x_{n-1},\cdots, -x_{2}, -x_{1}),\\
\om^{2}(x_{1},\cdots, x_{n})=(-x_{1},x_{2},\cdots, x_{n-1}, -x_{n}),\\
\om^{3}(x_{1},\cdots, x_{n})=(-x_{n},-x_{n-1},\cdots, -x_{2}, x_{1}).
\end{eqnarray*}
Write $n=2m+1$. We see that $W(\Om)$ is of type $D_{3}\times A_{1}^{m-1}\cong A_{3}\times A_{1}^{m-1}$, and $\om$ is $(C, s,\cdots, s)$ where $C\in S_{4}$ is a Coxeter element, and $s\in S_{2}$ is the nontrivial element. Let $p:\Om\to \ZZ/2\ZZ$ be the projection, then by Lemma \ref{l:prop sign},

\begin{equation*}
\e_{3}|_{\Om}\sim (m-1)p^{*}\e_{3}^{S_{2}}+\e_{3}^{S_{4}}|_{\Om}.
\end{equation*}
Note that $p^{*}$ induces zero on cohomology in degrees $>1$, therefore
\begin{equation*}
\e_{3}|_{\Om}\sim \e_{3}^{S_{4}}|_{\Om}
\end{equation*}
which we know is nontrivial by \S\ref{ss:A}.
 

The unique nontrivial subgroup of $\Om$ is $\j{\om^{2}}$. Since the restriction map $\cohog{3}{\ZZ/4\ZZ,\FF_{2}}\to \cohog{3}{\ZZ/2\ZZ,\FF_{2}}$ is trivial, $\e_{3}|_{\j{\om^{2}}}$ is cohomologically trivial. 

\subsection{Type $D_{n}$, $n$ even}
In this case $\Om\cong\ZZ/2\ZZ\times\ZZ/2\ZZ$. The three nontrivial elements in $i(\Om)\subset W$ are:
\begin{eqnarray*}
\om_{1}(x_{1},\cdots, x_{n})=(x_{n},-x_{n-1},\cdots, -x_{2}, x_{1}),\\
\om_{2}(x_{1},\cdots, x_{n})=(-x_{1},x_{2},\cdots, x_{n-1}, -x_{n})\\
\om_{3}(x_{1},\cdots, x_{n})=(-x_{n},-x_{n-1},\cdots, -x_{2}, -x_{1}).
\end{eqnarray*}
We identify $\Om$ with $\ZZ/2\ZZ\times\ZZ/2\ZZ$ by $\om_{1}\mapsto (1,0), \om_{2}\mapsto(1,1)$ and $\om_{3}\mapsto(0,1)$. Let $p_{1},p_{2}:\Om\to \ZZ/2\ZZ$ be the projection to the first and second factors and let $p_{3}=p_{1}+p_{2}:\Om\to \ZZ/2\ZZ$. 

Write $n=2m$. We see that $W(\Om)$ is of type $D_{2}\times A_{1}^{m-1}\cong  A_{1}^{m+1}$. We write $D_{2}$ for the first two factors of $A_{1}$. Then
\begin{equation*}
\om_{1}=(s,1,s,\cdots); \quad \om_{2}=(s,s,1,\cdots); \quad \om_{3}=(1,s, s,\cdots)
\end{equation*}
where $s$ denotes the nontrivial element in each $A_{1}$-factor of $W(\Om)$. Each factor of $A_{1}$ corresponds to a projection $\pi_{i}: \Om\to S_{2}=\ZZ/2\ZZ$, $i=1,\cdots, m+1$, then $\e_{3}|_{\Om}$ is cohomologous to $\sum_{i=1}^{m+1}\pi_{i}^{*}\e_{3}^{S_{2}}$. Since $\pi_{1}=p_{1}$, $\pi_{2}=p_{2}$, $\pi_{3}=\cdots =\pi_{m+1}=p_{3}$. We thus get
\begin{equation}\label{Dev}
\e_{3}|_{\Om}\sim p_{1}^{*}(\e_{3}^{S_{2}})+p_{2}^{*}(\e_{3}^{S_{2}})+(m-1)p_{3}^{*}(\e_{3}^{S_{2}}).
\end{equation}
Let $\y\in\cohog{1}{\ZZ/2\ZZ,\FF_{2}}$ be the nontrivial element, then $\{\y^{a}\ot \y^{3-a}|0\le a\le 3\}$ gives a basis for $\cohog{3}{\Om,\FF_{2}}$. Then \eqref{Dev} implies
\begin{equation*}
[\e_{3}|_{\Om}]=m(\y^{3}\ot1)+(m-1)(\y^{2}\ot \y)+(m-1)(\y\ot \y^{2})+m(1\ot \y^{3})\in \cohog{3}{\Om,\FF_{2}}.
\end{equation*}
In particular, $[\e_{3}|_{\Om}]$ is always nontrivial.

We have three nontrivial subgroups in $\Om$, generated by $\om_{i}$, $i=1,2,3$. We see that
\begin{equation*}
[\e_{3}|_{\j{\om_{1}}}]=m\y^{3}, \quad [\e_{3}|_{\j{\om_{2}}}]=0, \quad [\e_{3}|_{\j{\om_{3}}}]=m\y^{3}.
\end{equation*}


\subsection{Type $E_{7}$} In this case, $\Om=\ZZ/2\ZZ$, and the nontrivial element $\t\in \Om$ has $\dim V^{\t}=4$. Therefore, $W(\t)$ is a rank $3$ parabolic subgroup of $W$ and $\t$ is the longest element in $W(\t)$ that acts by $-1$ on the $3$-dimensional reflection representation of $W(\t)$. Now $W(\t)$ is simply-laced since $W$ is, it implies that $W(\t)$ is of type $A_{1}^{3}$, therefore $\t$ is conjugate to a product of three commuting simple reflections in $W$. By Lemma \ref{l:prod}(2),  $[\e_{3}|_{\Om}]\in \cohog{3}{\Om,\FF_{2}}=\FF_{2}$ is the nontrivial element.


\subsection*{Acknowledgement} 
The author would like to thank G.Lusztig for the collaboration on \cite{LY} from which the cocycle $\e_{3}^{W}$ was discovered, and for his helpful comments on the draft. He also thanks R.Bezrukavnikov, P.Etingof and D.Nadler for stimulating discussions.

\end{document}